\documentclass[11pt]{amsart}
\usepackage{mathrsfs}
\usepackage{amssymb}
\usepackage[T1]{fontenc}
\usepackage[all]{xy}
\usepackage{amscd}
\usepackage{amsmath, amssymb}
\usepackage{amsthm}
\usepackage{amsfonts}
\usepackage[colorlinks,linkcolor=blue,citecolor=blue, pdfstartview=FitH]
{hyperref}
\usepackage{backref}
\usepackage{amscd}
\usepackage{easybmat}
\usepackage{mathrsfs}
\usepackage{amsfonts}
\usepackage{color}
\usepackage{pifont}
\usepackage{upgreek}
\usepackage{bm}
\usepackage{hyperref}
\usepackage{shorttoc}
\usepackage{amsmath,amstext,amsthm,a4,amssymb,amscd}
\usepackage[mathscr]{eucal}
\usepackage{mathrsfs}
\usepackage{epsf}
\usepackage{tikz-cd}
\usepackage{CJK}

\def\p{\partial}

\def\b{\bar}

\def\wt{\widetilde}

\theoremstyle{plain}
\newtheorem{thm}{Theorem}[section]
\newtheorem{lemma}[thm]{Lemma}
\newtheorem{prop}[thm]{Proposition}
\newtheorem{cor}[thm]{Corollary}
\theoremstyle{definition}

\newtheorem{ex}[thm]{Example}
\theoremstyle{definition}
\newtheorem{defn}[thm]{Definition}
\numberwithin{equation}{section}

\usepackage{fancyhdr}
\makeatletter
\@namedef{subjclassname@2020}{\textup{2020} Mathematics Subject Classification}
\makeatother

\begin{document}
\begin{CJK}{UTF8}{gbsn}
\title{Obstructions of deforming complex structures and cohomology contractions}

\date{\today}

\author{Xueyuan Wan}

\author{Wei Xia}

\address{Xueyuan Wan: Mathematical Science Research Center, Chongqing University of Technology, Chongqing 400054, China}
\email{xwan@cqut.edu.cn}

\address{Wei Xia: Mathematical Science Research Center, Chongqing University of Technology, Chongqing, P. R. China, 400054}
\email{xiawei@cqut.edu.cn, xiaweiwei3@126.com}

\begin{abstract}
The Kodaira principle asserts that suitable cohomological contraction maps annihilate obstructions to deforming complex structures. In this paper, we revisit these phenomena from a purely analytic point of view, developing a refined power series method for the deformation of $(p,q)$-forms and complex structures. Working with the Fr\"olicher spectral sequence, we show that under natural partial vanishing conditions on its differentials, all obstruction classes lie in the kernel of the corresponding contraction maps. This yields a refined Kodaira principle that recovers and strictly extends the known results. As a main application, we obtain new unobstructedness criteria for compact complex manifolds with trivial canonical bundle.
\end{abstract}

\maketitle
\tableofcontents

\section{Introduction}
The deformation space can exhibit arbitrary singularities---a phenomenon often referred to as 'Murphy's Law' in algebraic geometry, as demonstrated by Vakil \cite{Vak06}. A fundamental problem, therefore, is to determine when the deformation space is smooth. To address this, a celebrated approach is available: the so-called Kodaira principle.
The Kodaira principle is a concise expression of the idea that the ambient cohomology of a K\"ahler manifold naturally annihilates obstructions to deformation. This viewpoint originates in the pioneering work of Kodaira (see, for example, \cite{MR2109686, Cle05, Man}). The principle was first formulated by Kodaira-Spencer \cite{KS59} in the study of deformations of complex hypersurfaces and was later extended by Bloch \cite{Blo72} to arbitrary codimension. Since then, analogous phenomena have been investigated in a wide range of deformation-theoretic contexts; see, for instance, \cite{Ran99a, Ran99b, Cle05, BF03, Man04,KKP08,Iac11}. 
A classical and particularly influential result is a theorem of Clemens \cite[Theorem~10.1]{Cle05}, which states that the obstruction classes to deforming the complex structure of a compact K\"ahler manifold $X$ are annihilated by a suitable cohomological contraction map:
\[
\{\text{obstructions}\}
\subseteq
\ker\Bigl\{
H^{0,2}_{\b{\p}}\bigl(X,T^{1,0}\bigr)
\longrightarrow
\bigoplus_{p,q}\operatorname{Hom}\bigl(H^{p,q}_{\b{\p}}(X),H^{p-1,q+2}_{\b{\p}}(X)\bigr)
\Bigr\}.
\]

To make this precise, recall that there is a natural cohomological contraction map
\begin{equation}\label{eqn1.1}
\mu\colon
H^{0,2}_{\b{\p}}\bigl(X,T^{1,0}\bigr)
\longrightarrow
\bigoplus_{p,q}\operatorname{Hom}\bigl(H_{\bar\partial}^{p,q}(X),H_{\bar\partial}^{p-1,q+2}(X)\bigr),
\quad
[\sigma]\longmapsto [i_\sigma(\bullet)].
\end{equation}
Clemens' theorem can be conveniently reformulated in terms of~\eqref{eqn1.1} as follows.

\begin{thm}[{\cite[Theorem~10.1]{Cle05}}]\label{thm1.1}
Let $X$ be a compact K\"ahler manifold and let $\sum_{1\le j\le N}\phi_j$ be an $N$-th order deformation of $X$. Then the obstruction of $\sum_{1\le j\le N}\phi_j$ lies in the kernel of the cohomological contraction map~\eqref{eqn1.1}, i.e.
\[
\sum_{j=1}^N[\phi_j,\phi_{N+1-j}] \in \ker\mu.
\]
In particular, if $\ker\mu=0$, then the deformations of $X$ are unobstructed.
\end{thm}

There are two main approaches to Theorem~\ref{thm1.1}.  
Clemens' original proof \cite{Cle05} relies on his theory of transversely holomorphic trivializations, which is technically quite involved.  
A second, more algebraic approach is based on differential graded Lie algebras and $L_\infty$-algebras, as developed in \cite{Man04,IM10,Man}.  
Within this latter framework, Theorem~\ref{thm1.1} has been substantially generalized to certain non-K\"ahler manifolds and is usually referred to as a \emph{Kodaira principle}. 
\begin{thm}[{\cite[Corollary~8.8.4]{Man}, Kodaira principle}]\label{thm1.2}
Let $X$ be a complex manifold and let $p$ be a positive integer such that the three inclusions
\[
F_X^{p+1} \longrightarrow F_X^p \longrightarrow F_X^{p-1} \longrightarrow F_X^0 = A_X^{*,*}
\]
are injective in cohomology. Then the contraction map
\[
\boldsymbol{i}\colon
H^{2}(X,T^{1,0})
\longrightarrow
\bigoplus_{q=0}^n\operatorname{Hom}\bigl(H^q(X,\Omega_X^p),H^{q+2}(X,\Omega_X^{p-1})\bigr),
\quad
\boldsymbol{i}_\eta(\omega)=\eta\lrcorner\omega,
\]
annihilates every obstruction to deformations of $X$.
\end{thm}

The main goal of this paper is threefold.

\smallskip

\noindent
(1) We first clarify the relationship between the Kodaira principle and our earlier work on deformations of $(p,q)$-forms \cite{WX2023}.  
In \cite{WX2023} we developed a systematic analytic framework for extending $(p,q)$-forms along deformations.  
Here we refine this viewpoint and show how the Kodaira-type phenomena naturally emerge from the finite-step extension machinery.

\smallskip

\noindent
(2) We then provide a new analytic perspective on the obstructions to deforming complex structures.  
Our arguments are based on the classical power series method combined with a careful analysis of the Fr\"olicher spectral sequence, in the spirit of \cite{WX2023,MR4425290,Xia19dDol,MR3920325,MR3994313,LRY15}.  
This allows us to isolate precisely which pieces of the Fr\"olicher differentials control the obstruction classes.

\smallskip

\noindent
(3) Finally, we weaken global cohomological assumptions such as the K\"ahler or $\partial\bar\partial$ conditions and replace them with finer conditions on the Fr\"olicher spectral sequence.  
This makes our results applicable to a wider class of non-K\"ahler manifolds, such as those considered in \cite{Ghy95,Rol11a,Liu16,ACRT18,Pop19}, while still retaining strong conclusions on unobstructedness.

\medskip

Let $X$ be a compact complex manifold of dimension $n$, and denote its Fr\"olicher spectral sequence by $(E_r^{p,q},d_r^{p,q})$, where
\[
d_r^{p,q}\colon E_r^{p,q}\longrightarrow E_r^{p+r,q-r+1}
\]
is the $r$-th differential (see Section~\ref{FSS} for details).  
We say that the Fr\"olicher spectral sequence degenerates at $E_1$ if $d_r^{p,q}=0$ for all $r\ge1$ and all $p,q$.  
In particular, any compact K\"ahler manifold satisfies the $\partial\bar{\partial}$-lemma, and hence its Fr\"olicher spectral sequence degenerates at $E_1$.  

For the purposes of controlling obstructions, however, it is not necessary to require the vanishing of \emph{all} differentials $d_r^{p,q}$.  
A crucial point of this paper is that suitable vanishing conditions restricted to a fixed bidegree $(p,q)$ (and a finite number of adjacent bidegrees) already suffice to force the obstruction map to land in the kernel of the contraction map.  
This leads to the following refinement of the Kodaira principle, simultaneously extending Theorems~\ref{thm1.1} and~\ref{thm1.2}.

\begin{thm}\label{thm1.3}
Let $X$ be a compact complex manifold such that
\[
\bigoplus_{r\ge 1} d_r^{p,q} = 0,
\quad
\bigoplus_{r\ge 1} d_r^{p-1,q+1} = 0,
\quad
\bigoplus_{\substack{r-1 \ge i\ge 0}} d_r^{p-2-i,\,q+2+i}(X)=0.
\]
Then the obstruction
\(
\sum_{j=1}^N [\phi_j,\phi_{N+1-j}]
\)
lies in the kernel of the contraction map
\[
\mu_{p,q}\colon
H^{0,2}_{\b{\p}}\bigl(X,T^{1,0}\bigr)
\longrightarrow
\operatorname{Hom}\bigl(H_{\bar\partial}^{p,q}(X),H_{\bar\partial}^{p-1,q+2}(X)\bigr),
\]
that is,
\[
\sum_{j=1}^N [\phi_j,\phi_{N+1-j}] \in \ker \mu_{p,q},
\]
where $\mu_{p,q}$ is the $(p,q)$-component of $\mu$.
\end{thm}

In other words, Theorem~\ref{thm1.3} may be viewed as a \emph{refined Kodaira principle} at the level of a fixed bidegree \((p,q)\), formulated entirely in terms of partial vanishing of the Fr\"olicher differentials; see Section~\ref{sec-KP} for a more detailed discussion.
Our proof uses an analytic power-series construction for $(p,q)$-forms and works equally well beyond the K\"ahler category. 

On the other hand, the hypotheses in Theorem~\ref{thm1.3} are formulated at a fixed bidegree $(p,q)$, and this is precisely what leads to a more refined description of the obstruction space. In contrast, the injectivity-in-cohomology assumption in Theorem~\ref{thm1.2} is not imposed at a fixed bidegree: by definition, the cohomology of $F_X^p$ is computed with respect to the total differential $d$, which mixes Hodge types and therefore does not preserve the index $q$.

\medskip

The unobstructedness problem for deformations of compact complex manifolds is a central theme in complex geometry.  
The celebrated Bogomolov--Tian--Todorov theorem \cite{Tia87,Tod89} asserts that any compact K\"ahler Calabi--Yau manifold has unobstructed deformations.  
This result has been generalised in several directions; in particular, the K\"ahler hypothesis can be weakened to the $\partial\bar{\partial}$-assumption (and even to $E_1=E_\infty$ assumption), see, for example, \cite[Theorem~1.5]{Voisin_HodgeTopology}.  

More recently, Popovici--Stelzig--Ugarte \cite{PSU} studied unobstructed deformations of Calabi--Yau page-$1$-$\partial\bar{\partial}$ manifolds under additional assumptions, providing a unified framework that encompasses classical examples such as the $3$- and $5$-dimensional Iwasawa manifolds.  
In this setting, the Fr\"olicher spectral sequence is known to degenerate at $E_2$ \cite[Theorem and Definition~2.11]{PSU}.  
On the other hand, Popovici \cite{MR3978322} obtained a more general unobstructedness criterion formulated in terms of Bott--Chern and Aeppli cohomologies.

\begin{thm}[{\cite[Observation~3.5]{MR3978322}}]\label{thm1.4}
Let $X$ be a compact complex manifold of complex dimension $n$ with trivial canonical bundle $K_X$ such that the linear maps
\[
A_1\colon H_{\bar{\partial}}^{n-1,1}(X,\mathbb{C}) \longrightarrow H_{BC}^{n,1}(X,\mathbb{C}), 
\qquad
[\alpha]_{\bar{\partial}} \longmapsto [\partial\alpha]_{BC},
\]
and
\[
A_2\colon H_A^{n-2,2}(X,\mathbb{C}) \longrightarrow H_{BC}^{n-1,2}(X,\mathbb{C}),
\qquad
[v]_A \longmapsto [\partial v]_{BC},
\]
are identically zero. Then the Kuranishi family of $X$ is unobstructed.
\end{thm}

In the Calabi--Yau case, Theorem~\ref{thm1.3} admits a particularly clean formulation.  
If $X$ is a compact complex manifold of dimension $n$ with trivial canonical bundle, then there exists a nowhere vanishing holomorphic $(n,0)$-form $\Omega\in H_{\bar\partial}^{n,0}(X)$.  
Contraction with $\Omega$ provides an isomorphism
\[
\Omega\colon
H_{\bar\partial}^{0,2}\bigl(X,T^{1,0}\bigr)
\longrightarrow
H_{\bar\partial}^{n-1,2}(X),
\qquad
[\sigma]\longmapsto[i_\sigma(\Omega)].
\]
In particular, $\ker\mu_{n,0}=\{0\}$.  
Combining this observation with Theorem~\ref{thm1.3}, we obtain the following unobstructedness result.

\begin{thm}\label{thm1.5}
Let $X$ be a compact complex manifold of dimension $n$ with trivial canonical bundle. Suppose that
\[
\bigoplus_{r\ge 1} d_r^{n-1,1} = 0
\quad\text{and}\quad
\bigoplus_{\substack{r-1 \ge i\ge 0}} d_r^{n-2-i,\,2+i}(X)=0.
\]
Then the deformations of the complex structure on $X$ are unobstructed.
\end{thm}

We show in Lemma~\ref{lemma4.6} that the conditions $A_1=0=A_2$ in Theorem~\ref{thm1.4} are strictly stronger than the Fr\"olicher-type vanishing assumptions appearing in Theorem~\ref{thm1.5}.  
Thus Theorem~\ref{thm1.5} indeed provides a genuine extension of Popovici's unobstructedness theorem.  
As an immediate corollary, we obtain a particularly transparent criterion in terms of full degeneracy in total degree $n$.

\begin{cor}
Let $X$ be a compact complex manifold of dimension $n$ with trivial canonical bundle.  
If its Fr\"olicher spectral sequence degenerates at $E_1$ in total degree $n$, that is,
\[
d_r^{p,q}=0
\quad\text{for all }r\ge1\text{ and all }p,q\text{ with }p+q=n,
\]
then $X$ has unobstructed deformations.
\end{cor}

For solvable complex parallelisable manifolds, it was shown by Kasuya that the Fr\"olicher spectral sequence degenerates at $E_2$ \cite{Kas15}. In this case, our assumptions in Theorem \ref{thm1.5} simplify to
\[
d_1^{n-1,1}=0
\quad\text{and}\quad
d_1^{n-2,2}=0.
\]
These may be used to analyze the deformation obstructions of the complex parallelizable Nakamura manifold, see Example \ref{ex-Naka}.

In summary, our approach provides a refined, degree-wise Kodaira principle formulated in terms of the Fr\"olicher spectral sequence, and yields new unobstructedness results for Calabi-Yau manifolds beyond the reach of existing criteria.

The article is organized as follows.  
In Section~\ref{FSS}, we review some basic facts about the Fr\"olicher spectral sequence, including several equivalent characterizations of the vanishing of certain differentials.  
In Section~\ref{finite-steps}, we recall the deformation theory of \((p,q)\)-forms developed in \cite{WX2023}, with particular emphasis on the finite-step extension procedure.  
In Section~\ref{Obstructions}, we study obstructions to deforming complex structures and criteria for unobstructedness, and we give the proofs of Theorems~\ref{thm1.3} and~\ref{thm1.5}. In Section~\ref{sec-para}, we study the unobstructedness of deformations of complex parallelisable manifolds.

\section{Fr\"olicher spectral sequence}\label{FSS}

In this section, using the general description of the Fr\"olicher spectral sequence due to Cordero--Fern\'andez--Ugarte--Gray \cite[Theorem~1 and Theorem~3]{CFUG}, we recall an equivalent characterization of the degeneration of the Fr\"olicher spectral sequence at the first page.

Let $X$ be a compact complex manifold of dimension $n$. By definition, the Fr\"olicher spectral sequence of $X$ is the spectral sequence associated with the filtered complex
\[
\bigl(F^\bullet A^\bullet(X),d\bigr), 
\qquad 
F^pA^k(X):=\bigoplus_{\lambda\ge p}A^{\lambda,k-\lambda}(X).
\]
Set
\[
Z^{p,q}_r
:=F^pA^{p+q}\cap d^{-1}\bigl(F^{p+r}A^{p+q+1}\bigr),
\qquad
B^{p,q}_r
:=F^pA^{p+q}\cap d\bigl(F^{p-r}A^{p+q-1}\bigr),
\]
where, for brevity, we write $F^\bullet A^\bullet:=F^\bullet A^\bullet(X)$.  
Then the Fr\"olicher spectral sequence is given by
\[
E_r^{p,q}(X)
=Z^{p,q}_r\big/\bigl(Z^{p+1,q-1}_{r-1}+B^{p,q}_{r-1}\bigr).
\]

By \cite[Theorem~1]{CFUG}, the groups $E_r^{p,q}(X)$ admit the description
\[
E_r^{p,q}(X)=\frac{\widetilde{Z}_r^{p,q}(X)}{\widetilde{B}_r^{p,q}(X)},
\]
where
\[
\widetilde{Z}_1^{p,q}(X)=A^{p,q}(X)\cap\ker\bar\partial,
\qquad
\widetilde{B}_1^{p,q}(X)=\bar\partial\bigl(A^{p,q-1}(X)\bigr),
\]
and, for $r\ge2$,
\begin{align*} \begin{split} \wt{Z}^{p,q}_r(X)=&\{\alpha^{p,q}\in A^{p,q}(X)|\b{\p}\alpha^{p,q}=0 \text{ and there exist }\\& \alpha^{p+i,q-i}\in A^{p+i,q-i}(X) \text{ such that }\\ & \p\alpha^{p+i-1,q-i+1}+\b{\p}\alpha^{p+i,q-i}=0,\,1\leq i\leq r-1\} \end{split} \end{align*} \begin{align*} \begin{split} \wt{B}^{p,q}_r(X)=&\{\p\beta^{p-1,q}+\b{\p}\beta^{p,q-1}\in A^{p,q}(X)|\text{ there exist}\\ & \beta^{p-i,q+i-1}\in A^{p-i,q+i-1}(X),\,2\leq i\leq r-1,\\ &\text{ satisfying } \p\beta^{p-i,q+i-1}+\b{\p}\beta^{p-i+1,q+i-2}=0,\\ &\b{\p}\beta^{p-r+1,q+r-2}=0\}. \end{split} \end{align*}In particular,
\[
E_1^{p,q}(X)\cong
\frac{A^{p,q}(X)\cap\ker\bar\partial}{\bar\partial\bigl(A^{p,q-1}(X)\bigr)}
=H^{p,q}_{\bar\partial}(X),
\]
and
\[
E_2^{p,q}(X)
\cong
\frac{\{\alpha^{p,q}\in A^{p,q}(X)\mid 0=\bar\partial\alpha^{p,q}=\partial\alpha^{p,q}+\bar\partial\alpha^{p+1,q-1}\}}
{\{\partial\beta^{p-1,q}+\bar\partial\beta^{p,q-1}\mid 0=\bar\partial\beta^{p-1,q}\}}.
\]

By \cite[Theorem~3]{CFUG}, for $r\ge2$ the differential
\[
d_r^{p,q}\colon E_r^{p,q}(X)\longrightarrow E_r^{p+r,q-r+1}(X)
\]
is given by
\begin{equation}\label{d-operator}
  d_r^{p,q}[\alpha^{p,q}] = [\,\partial\alpha^{p+r-1,q-r+1}\,],
\end{equation}
for $[\alpha^{p,q}]\in E_r^{p,q}(X)$. Moreover,
\[
E_{r+1}^{p,q}(X)
\equiv \frac{\widetilde{Z}_{r+1}^{p,q}(X)}{\widetilde{B}_{r+1}^{p,q}(X)}
\equiv
\frac{\ker\bigl(d_r\colon E_r^{p,q}(X)\to E_r^{p+r,q-r+1}(X)\bigr)}
{d_r\bigl(E_r^{p-r,q+r-1}(X)\bigr)}.
\]
For $r=1$ we have $E_1^{p,q}(X)\equiv H^{p,q}_{\bar\partial}(X)$, and the map
\[
d_1^{p,q}\colon E_1^{p,q}(X)\longrightarrow E_1^{p+1,q}(X)
\]
is induced by
\[
\partial\colon H^{p,q}_{\bar\partial}(X)\longrightarrow H^{p+1,q}_{\bar\partial}(X).
\]

\begin{defn}
We say that the Fr\"olicher spectral sequence $\{E_r,d_r\}$ degenerates at $E_1$ if 
\[
d_r^{p,q}\colon E_r^{p,q}(X)\longrightarrow E_r^{p+r,q-r+1}(X)
\]
vanishes for all $r\ge1$ and all $p,q$.
\end{defn}

If the Fr\"olicher spectral sequence $\{E_r,d_r\}$ degenerates at $E_1$, then
\[
E_1^{p,q}\cong E_2^{p,q}\cong\cdots\cong E_\infty^{p,q}
=\mathrm{gr}^p_F H^{p+q}_{\mathrm{dR}}(X)
\quad\text{for all }p,q.
\]

Similarly, for fixed bidegrees or total degree we have the following notions.

\begin{defn}
We say that the Fr\"olicher spectral sequence $\{E_r,d_r\}$ degenerates at $E_{r_0}$ \emph{in bidegree $(p,q)$} if
\[
d_r^{p,q}\colon E_r^{p,q}(X)\longrightarrow E_r^{p+r,q-r+1}(X)
\]
vanishes for all $r\geq r_0$.

We say that $\{E_r,d_r\}$ degenerates at $E_{r_0}$ \emph{in total degree $k$} if
\[
d_r^{p,q}\colon E_r^{p,q}(X)\longrightarrow E_r^{p+r,q-r+1}(X)
\]
vanishes for all $r\geq r_0$ and all $p,q$ with $p+q=k$.
\end{defn}

For the degeneration at $E_1$ in a fixed bidegree $(p,q)$, we have the following characterization.

\begin{lemma}[{\cite[Proposition~B.3]{WX2023}}]\label{cordeg1}
The following are equivalent:
\begin{itemize}
  \item[(i)] The Fr\"olicher spectral sequence degenerates at $E_1$ for $(p,q)$-forms, i.e.
  \[
    d_r^{p,q}=0 \quad \text{for all } r\ge1.
  \]
  \item[(ii)] For every $\alpha^{p,q}\in A^{p,q}(X)\cap\ker\bar\partial$, there exists $x\in \ker d\cap F^pA^{p+q}(X)$ such that $x^{p,q}=\alpha^{p,q}$.
  \item[(iii)] One has
  \[
    F^{p+1}A^{p+1+q}(X)\cap dF^pA^{p+q}(X)
    = dF^{p+1}A^{p+q}(X).
  \]
\end{itemize}
\end{lemma}
We also have the following equivalence.
\begin{lemma}[{\cite[Corollary B.6]{WX2023}}]\label{cordeg2}
The following are equivalent:
\begin{itemize}
  \item[(i)] For every $r\geq 1$ we have
  \[
    \bigoplus_{r-1 \geq i \geq 0} d_r^{p-i,\,q+i}(X) = 0.
  \]
  \item[(ii)] We have
  \[
    F^{p+1}A^{p+q+1}(X)\cap dA^{p+q}(X)
    \;=\;
    dF^{p+1}A^{p+q}(X).
  \]
\end{itemize}
\end{lemma}

\section{Deformations of $(p,q)$-forms for finite steps}\label{finite-steps}

In this section we describe a finite-step extension procedure for \((p,q)\)-forms.  
A systematic treatment of extensions of \((p,q)\)-forms was given in our previous work \cite{WX2023}; here we present a slightly simplified argument, focusing on the finite-step extensions needed in the present paper.

Let \(\pi\colon \mathcal{X} \to \Delta\) be a holomorphic family of compact complex manifolds with central fibre \(X=X_0\), where \(\dim X_0 = n\).  
Let \(\phi = \phi(t) \in A^{0,1}(X_0,T^{1,0})\) be the Beltrami differential associated with the complex structure of \(\mathcal{X}_t\), \(t \in \Delta\).  
By construction, \(\phi(t)\) depends holomorphically on \(t\) and satisfies \(\phi(0)=0\).  
We are naturally led to consider a \(\bar{\partial}\)-closed \((p,q)\)-form \(\alpha_0^{p,q} \in A^{p,q}(X_0)\cap\ker\bar{\partial}\) on \(X_0\), and ask under what conditions it admits an unobstructed extension to nearby fibres, i.e.\ whether we can find a continuous family
\[
\beta(t)\in A^{p,q}(X_t)\cap\ker\bar{\partial}_t,\qquad t\in\Delta,
\]
with initial condition \(\beta(0)=\alpha_0^{p,q}\).

Our approach is as follows.  
For \(|t|\) sufficiently small, the operators
\[
e^{i_{(1-\bar\phi\phi)^{-1}\bar\phi}}
\quad\text{and}\quad
e^{-i_\phi}
\]
are invertible, where \(e^{-i_\phi} = \sum_{k=0}^{\infty}\frac{(-i_\phi)^k}{k!}\) is the inverse of \(e^{i_\phi}\).  
For any \(\widetilde{\alpha}\in A^{p,q}(X_0)\), we have
\[
e^{i_\phi\mid i_{\bar\phi}}(\widetilde{\alpha})
=
e^{i_\phi}\circ
e^{i_{(1-\bar\phi\phi)^{-1}\bar\phi}}\circ
e^{-i_{(1-\bar\phi\phi)^{-1}\bar\phi}}\circ
e^{-i_\phi}\circ
e^{i_\phi\mid i_{\bar\phi}}(\widetilde{\alpha}).
\]
Set
\[
\alpha
:=
e^{-i_{(1-\bar\phi\phi)^{-1}\bar\phi}}\circ
e^{-i_\phi}\circ
e^{i_\phi\mid i_{\bar\phi}}(\widetilde{\alpha}).
\]
Since \(e^{i_{(1-\bar\phi\phi)^{-1}\bar\phi}}\), \(e^{i_\phi}\) and \(e^{i_\phi\mid i_{\bar\phi}}\) are invertible for \(|t|\) small, the correspondence between \(\widetilde{\alpha}\) and \(\alpha\) is bijective.  
Moreover, it is straightforward to check that the operator
\[
e^{-i_{(1-\bar\phi\phi)^{-1}\bar\phi}}\circ
e^{-i_\phi}\circ
e^{i_\phi\mid i_{\bar\phi}}
\]
preserves the bidegree of forms, so that \(\alpha\) is again a \((p,q)\)-form on \(X_0\).  
Indeed, for any \((p,q)\)-form \(\widetilde{\alpha}\) on \(X_0\), we obtain
\begin{align*}
\begin{split}
e^{-i_{(1-\bar\phi\phi)^{-1}\bar\phi}}\circ e^{-i_\phi}\circ e^{i_\phi\mid i_{\bar\phi}}(\widetilde{\alpha})
&=
\widetilde{\alpha}_{i_1\cdots i_p\bar{j}_1\cdots\bar{j}_q}\,
dz^{i_1}\wedge\cdots\wedge dz^{i_p}\\
&\phantom{=} \wedge
(1-\bar\phi\phi)\lrcorner\, d\bar z^{j_1}\wedge\cdots\wedge
(1-\bar\phi\phi)\lrcorner\, d\bar z^{j_q}
\in A^{p,q}(X_0),
\end{split}
\end{align*}
where
\[
\widetilde{\alpha}
=
\widetilde{\alpha}_{i_1\cdots i_p\bar{j}_1\cdots\bar{j}_q}\,
dz^{i_1}\wedge\cdots\wedge dz^{i_p}\wedge
d\bar z^{j_1}\wedge\cdots\wedge d\bar z^{j_q}.
\]
Thus \(\alpha\) is still of type \((p,q)\), since \((1-\bar\phi\phi)\lrcorner\) preserves bidegree.

Denote
\[
\bar{\partial}_\phi := \bar{\partial} + [\partial,i_\phi].
\]
From \cite[(20)]{MR4425290}, we have
\[
\bar{\partial}_t\bigl(e^{i_\phi\mid i_{\bar\phi}}(\widetilde{\alpha})\bigr)
=
e^{i_\phi\mid i_{\bar\phi}}
\bigl((1-\bar\phi\phi)^{-1}\,\Finv\bar{\partial}_\phi\alpha\bigr).
\]
Hence \(\bar{\partial}_t\bigl(e^{i_\phi\mid i_{\bar\phi}}(\widetilde{\alpha})\bigr)=0\) if and only if
\begin{equation}\label{eq:phi-equation}
\bar{\partial}_\phi \alpha = 0.
\end{equation}

Our goal is therefore to find a family \(\alpha^{p,q}(t)\in A^{p,q}(X_0)\) solving
\begin{equation}\label{obs1}
\bar{\partial}_\phi \alpha^{p,q}(t) = 0,
\qquad
\alpha^{p,q}(0) = \alpha^{p,q}_0,
\end{equation}
for any \(\alpha^{p,q}_0\in A^{p,q}(X_0)\cap\ker\bar{\partial}\).  
If we can solve \eqref{obs1} for \(\alpha^{p,q}(t)\), then
\[
\widetilde{\alpha}^{p,q}(t)
:=
e^{-i_\phi\mid i_{\bar\phi}}\circ
e^{i_\phi}\circ
e^{i_{(1-\bar\phi\phi)^{-1}\bar\phi}}(\alpha^{p,q}(t))
\in A^{p,q}(X_0),
\]
and
\[
\beta(t)=e^{i_\phi\mid i_{\bar\phi}}\bigl(\widetilde{\alpha}^{p,q}(t)\bigr)
\in A^{p,q}(X_t)\cap\ker\bar{\partial}_t
\]
provides a smooth extension of \(\alpha_0^{p,q}\) to the nearby fibres \(X_t\).

In \cite{WX2023}, we carried out a systematic study of the solvability of \eqref{obs1}.  
Here we are interested in a finite-order version of that problem.

Assume that
\[
  \phi = \sum_{i=1}^N \phi_i
\]
and that $\phi$ satisfies the integrability condition
\[
  \bar{\partial}\phi = \tfrac{1}{2}[\phi,\phi].
\]
We seek a family $\alpha^{p,q}(t)\in A^{p,q}(X_0)$ such that
\begin{equation}\label{obs11}
  \bar{\partial}_\phi \alpha^{p,q}(t) = 0 \ \text{mod } t^{N+1},
  \qquad
  \alpha^{p,q}(0) = \alpha^{p,q}_0,
\end{equation}
for a given $\alpha^{p,q}_0\in A^{p,q}(X_0)\cap\ker\bar{\partial}$.

To solve \eqref{obs11}, we look for
\[
  \alpha(t)
  = \alpha^{p,q}(t) + \cdots + \alpha^{p+n,q-n}(t)
  \in F^p A^{p+q}(X_0),
\]
holomorphic in $t$, such that
\begin{equation}\label{obs44}
  d\bigl(e^{i_\phi}(\alpha(t))\bigr) = 0 \ \text{mod } t^{N+1},
  \qquad
  \alpha^{p,q}(0) = \alpha^{p,q}_0.
\end{equation}
Indeed, if \eqref{obs44} holds, then using the identity
\[
  d \circ e^{i_\phi}
  = e^{i_\phi}\bigl(\bar{\partial}_\phi + \partial
     + i_{\bar{\partial}\phi - \frac{1}{2}[\phi,\phi]}\bigr),
\]
and the integrability condition $\bar{\partial}\phi = \tfrac{1}{2}[\phi,\phi]$, we obtain that \eqref{obs44} is equivalent to
\[
  (\bar{\partial}_\phi + \partial)\,\alpha(t) = 0 \ \text{mod } t^{N+1}.
\]
Writing $\alpha(t)$ by type as above, this is equivalent to the triangular system
\begin{align*}
\begin{split}
  \begin{cases}
    \bar{\partial}_\phi \alpha^{p,q}(t) &= 0 \ \text{mod } t^{N+1},\\[2pt]
    \bar{\partial}_\phi \alpha^{p+1,q-1}(t) + \partial\alpha^{p,q}(t)
      &= 0 \ \text{mod } t^{N+1},\\
    \quad\vdots\\
    \bar{\partial}_\phi \alpha^{p+n,q-n}(t)
      + \partial \alpha^{p+n-1,q-n+1}(t)
      &= 0 \ \text{mod } t^{N+1},
  \end{cases}
\end{split}
\end{align*}
and the first equation is precisely \eqref{obs11}.

We now expand $\alpha(t)$ as a finite power series
\[
  \alpha(t) = \sum_{k=0}^N \alpha_k,
\]
where $\alpha_k$ denotes the homogeneous term of order $k$ in $t$.  
For later use, we also set
\[
  \alpha_{\le N_0} := \sum_{k=0}^{N_0} \alpha_k.
\]

\begin{lemma}\label{lemma11}
For any $N_0\le N$ one has
\[
\bigl(d\circ e^{i_\phi}(\alpha(t))\bigr)_{\le N_0}=0
\quad\Longleftrightarrow\quad
\bigl((\bar\partial_\phi+\partial)\alpha(t)\bigr)_{\le N_0}=0.
\]
\end{lemma}

\begin{proof}
By the identity
\[
e^{-i_\phi}\circ d\circ e^{i_\phi}
= \bar\partial_\phi + \partial
\quad\text{mod } t^{N+1},
\]
we obtain
\[
e^{-i_\phi}\circ d\circ e^{i_\phi}(\alpha(t))
= (\bar\partial_\phi + \partial)\alpha(t)
\quad\text{up to terms of order } \ge N+1.
\]

If $\bigl(d\circ e^{i_\phi}(\alpha(t))\bigr)_{\le N_0}=0$, then applying $e^{-i_\phi}$ and truncating to order $\le N_0$ gives
\[
0
= \bigl(e^{-i_\phi}\circ d\circ e^{i_\phi}(\alpha(t))\bigr)_{\le N_0}
= \bigl((\bar\partial_\phi + \partial)\alpha(t)\bigr)_{\le N_0}.
\]

Conversely, if
\[
\bigl((\bar\partial_\phi + \partial)\alpha(t)\bigr)_{\le N_0}=0,
\]
then applying $e^{i_\phi}$ and truncating yields
\[
0
= \bigl(e^{i_\phi}(\bar\partial_\phi + \partial)\alpha(t)\bigr)_{\le N_0}
= \bigl(d\circ e^{i_\phi}(\alpha(t))\bigr)_{\le N_0}.
\]
This proves the equivalence.
\end{proof}
In the following, we will use $\alpha=\alpha(t)$ for simplicity, and denote by $\bullet^{p,q}$ the $(p,q)$-component of $\bullet\in A^*(X)$.
\begin{lemma}\label{lemma2.2}
We have
\begin{itemize}
  \item[(i)]If $(d(e^{i_\phi}(\alpha)))_{\leq N-1}=0$, then 
	\begin{equation}\label{equ3}
 \Pi^{\leq p-1,*} ( d\circ (e^{i_{\phi}}-1)(\alpha))_{\leq N}=0.
\end{equation}
\item[(ii)]If $(d(e^{i_\phi}(\alpha)))_{\leq N}=0$, one has
\begin{equation}\label{equ33}
 \Pi^{\leq p-2,*} ( d\circ (e^{i_{\phi}}-1)(\alpha))_{N+1}=0.
\end{equation}
\end{itemize}
\end{lemma}
\begin{proof}
	Note that 
\begin{align*}
\begin{split}
  d\circ (e^{i_{\phi}}-1)(\alpha)&=\left(e^{i_{\phi}}(\b{\p}_\phi+\p)-d\right)\alpha+e^{i_{\phi}}\circ i_{\b{\p}\phi-\frac{1}{2}[\phi,\phi]}\alpha\\
  &=(e^{i_{\phi}}-1)(\b{\p}_\phi+\p)\alpha+[\p,i_{\phi}]\alpha+e^{i_{\phi}}\circ i_{\b{\p}\phi-\frac{1}{2}[\phi,\phi]}\alpha.
 \end{split}
\end{align*}
If $(d(e^{i_\phi}(\alpha)))_{\leq N_0}=0$, by Lemma \ref{lemma11}, then 
$$((\b{\p}_\phi+\p)\alpha)_{\leq N_0}=0.$$
For any $j\geq 1$ and $N_0\leq N$, one has
\begin{align}\label{equ5}
\begin{split}
  & ( d\circ (e^{i_{\phi}}-1)(\alpha))_{\leq N_0+1}^{p-j,q+j+1}\\
   &=\left((e^{i_{\phi}}-1)(\b{\p}_\phi+\p)\alpha\right)_{\leq N_0+1}^{p-j,q+j+1}+\left([\p,i_{\phi}]\alpha\right)_{\leq N_0+1}^{p-j,q+j+1}\\
   &\quad +(e^{i_{\phi}}\circ i_{\b{\p}\phi-\frac{1}{2}[\phi,\phi]}\alpha)_{\leq N_0+1}^{p-j,q+j+1}\\
   &=\left((e^{i_{\phi}}-1)(\b{\p}_\phi+\p)\alpha\right)_{\leq N_0+1}^{p-j,q+j+1}
  +(e^{i_{\phi}}\circ i_{\b{\p}\phi-\frac{1}{2}[\phi,\phi]}\alpha)_{\leq N_0+1}^{p-j,q+j+1}
 \end{split}
\end{align}
since $([\p,i_{\phi}]\alpha)_{N_0+1}\in F^pA^{p+q+1}(X)$.
\begin{itemize}
  \item  If $(d(e^{i_\phi}(\alpha)))_{\leq N-1}=0$, taking $N_0=N-1$ and $j\geq 1$, we have
  \begin{equation*}
  \left((e^{i_{\phi}}-1)(\b{\p}_\phi+\p)\alpha\right)_{\leq N_0+1}^{p-j,q+j+1}
  =0=(e^{i_{\phi}}\circ i_{\b{\p}\phi-\frac{1}{2}[\phi,\phi]}\alpha)_{\leq N_0+1}^{p-j,q+j+1},
\end{equation*}
and so 
\begin{equation*}
   \Pi^{\leq p-1,*} ( d\circ (e^{i_{\phi}}-1)(\alpha))_{\leq N}=\sum_{j\geq 1}( d\circ (e^{i_{\phi}}-1)(\alpha))_{\leq N_0+1}^{p-j,q+j+1}=0.
\end{equation*}
\item If $(d(e^{i_\phi}(\alpha)))_{\leq N}=0$, taking $N_0=N$ and $j\geq 2$, we also have
   \begin{equation*}
  \left((e^{i_{\phi}}-1)(\b{\p}_\phi+\p)\alpha\right)_{\leq N_0+1}^{p-j,q+j+1}
  =0=(e^{i_{\phi}}\circ i_{\b{\p}\phi-\frac{1}{2}[\phi,\phi]}\alpha)_{\leq N_0+1}^{p-j,q+j+1},
\end{equation*}
and so 
\begin{equation*}
   \Pi^{\leq p-2,*} ( d\circ (e^{i_{\phi}}-1)(\alpha))_{\leq N+1}=\sum_{j\geq 2}( d\circ (e^{i_{\phi}}-1)(\alpha))_{\leq N_0+1}^{p-j,q+j+1}=0.
\end{equation*}
\end{itemize}
The proof is complete.
\end{proof}

\begin{thm}\label{thm11}
Let $X$ be a compact complex manifold, and suppose that
\[
\bigoplus_{r\ge 1} d_r^{p,q}=0
\quad\text{and}\quad
\bigoplus_{\substack{r-1 \ge i\ge 0}} d_r^{p-1-i,\,q+1+i}(X)=0.
\]
Then, for any $\alpha_0^{p,q}\in A^{p,q}(X)\cap \ker\bar\partial$, the equation
\begin{equation}\label{obs22}
  d\bigl(e^{i_\phi}(\alpha(t))\bigr)=0 \;\text{\emph{mod}}\; t^{N+1},
  \qquad
  \alpha(0)^{p,q}=\alpha_0^{p,q},
\end{equation}
admits a smooth solution $\alpha(t)=\sum_{k=0}^N \alpha_k \in F^pA^{p+q}(X)$.
\end{thm}

\begin{proof}
We argue by induction on the order $N_0$, $0\le N_0\le N$, solving \eqref{obs22} up to order $N_0$.

\smallskip\noindent
\textbf{Step $0$.}
For $N_0=0$, \eqref{obs22} is equivalent to
\[
d\alpha(0)=0,
\qquad
\alpha(0)^{p,q}=\alpha_0^{p,q}.
\]
Since $\alpha_0^{p,q}\in A^{p,q}(X)\cap\ker\bar\partial$, the assumption $\bigoplus_{r\ge 1} d_r^{p,q}=0$ together with Lemma~\ref{cordeg1} guarantees the existence of $\alpha(0)\in F^pA^{p+q}(X)$ such that $\alpha(0)^{p,q}=\alpha_0^{p,q}$ and $d\alpha(0)=0$.

\smallskip\noindent
\textbf{Inductive step.}
Assume that for some $0\le N_0\le N-1$ we have already constructed
\[
\alpha_{\le N_0}(t):=\sum_{k=0}^{N_0}\alpha_k \in F^pA^{p+q}(X)
\]
such that
\[
d\bigl(e^{i_\phi}(\alpha_{\le N_0}(t))\bigr)=0 \;\text{\emph{mod}}\; t^{N_0+1}
\quad\text{and}\quad
\alpha_0^{p,q}=\alpha(0)^{p,q}.
\]

\smallskip\noindent
\textbf{Step $N_0+1$.}
We now solve \eqref{obs22} at order $N_0+1$.  
Writing $\alpha(t)=\alpha_{\le N_0}(t)+\alpha_{N_0+1}$ and extracting the homogeneous component of order $N_0+1$, the condition \eqref{obs22} becomes
\begin{equation}\label{stepN+1}
  d\bigl(\alpha_{N_0+1}+((e^{i_\phi}-1)\alpha(t))_{N_0+1}\bigr)=0.
\end{equation}
By Lemma~\ref{lemma2.2} (i), the term
\[
d\bigl((e^{i_\phi}-1)\alpha(t)\bigr)_{N_0+1}
\]
belongs to
\[
dA^{p+q}(X)\cap F^pA^{p+q+1}(X).
\]
By the assumption
\[
\bigoplus_{\substack{r-1 \ge i\ge 0}} d_r^{p-1-i,\,q+1+i}(X)=0
\]
and Lemma~\ref{cordeg2}, we have
\[
F^pA^{p+q+1}(X)\cap dA^{p+q}(X)
=
dF^pA^{p+q}(X).
\]
Hence
\[
d\bigl((e^{i_\phi}-1)\alpha(t)\bigr)_{N_0+1}
\in dF^pA^{p+q}(X),
\]
so there exists $\alpha_{N_0+1}\in F^pA^{p+q}(X)$ solving \eqref{stepN+1}. This completes the inductive step.

By induction on $N_0$ from $0$ to $N$, we obtain the desired smooth solution $\alpha(t)=\sum_{k=0}^N\alpha_k\in F^pA^{p+q}(X)$ satisfying \eqref{obs22}.
\end{proof}

\section{Obstructions and the cohomology contractions}\label{Obstructions}

In this section, we briefly review the obstructions to deforming complex structures and the cohomology contraction map. Under suitable vanishing assumptions on the differentials in the Fr\"olicher spectral sequence, we show that these obstructions lie in the kernel of the contraction map. As an application, we generalize these Kodaira principles in \cite{Man} to the level of fixed bidegree \((p,q)\); in comparison with \cite{Man}, our analytic approach yields a more refined Kodaira-type principle.

\subsection{Obstructions to deforming complex structures}

For the classical deformation theory of complex manifolds, we refer the reader to \cite{MR2109686}.  
Following \cite{MR3184164}, a deformation of a compact complex space $X$ is given by a flat proper morphism
\[
\pi \colon (\mathcal{X}, X_0) \longrightarrow (\mathcal{B},0)
\]
between connected complex spaces, together with an isomorphism $\psi \colon X \to X_0$, where $X_0 := \pi^{-1}(0)$.  
A germ of such a deformation is called a \emph{small deformation}.  
In this paper we only consider small deformations of compact complex manifolds.  
The Kuranishi family
\[
\pi \colon (\mathcal{X}, X_0) \longrightarrow (\mathcal{B},0)
\]
is the unique (up to isomorphism) semiuniversal small deformation of $X$.

Let $X$ be a compact complex manifold.  
We say that the deformations of complex structures on $X$ are \emph{unobstructed} if the Kuranishi space $\mathcal{B}$ is smooth.  
By the classical Kodaira-Spencer-Kuranishi theory, this is equivalent to the following.  
Consider the Beltrami differential
\[
\phi(t) = \phi_1 + \frac{1}{2}\,\bar{\partial}^* G[\phi(t),\phi(t)],
\qquad
\phi_1 = \sum_{\nu=1}^m \eta_\nu t_\nu,
\]
where $\bar{\partial}^*$, $G$, and $\mathcal{H}$ are the usual operators from Hodge theory (in particular, $\mathcal{H}$ denotes the projection onto harmonic forms), and $\eta_1,\dots,\eta_r$ is a basis of
\[
\mathcal{H}^{0,l}(X,T^{1,0}),
\]
the space of harmonic $(0,1)$-forms with values in $T^{1,0}X$.  
The parameter $t = (t_1,\dots,t_r)$ is taken in a small neighbourhood of the origin in $\mathbb{C}^r$, and we write $\phi(t) = \sum_{k\ge 1} \phi_k(t)$, where $\phi_k$ is homogeneous of degree $k$ in $t$.

Then the deformations of complex structures on $X$ are unobstructed if and only if $\phi(t)$ satisfies the Maurer-Cartan equation
\begin{equation}\label{int}
\bar{\partial}\phi(t) - \frac{1}{2}[\phi(t),\phi(t)] = 0,
\end{equation}
or, equivalently,
\[
\mathcal{H}[\phi(t),\phi(t)] = 0
\]
for all sufficiently small $t$.  
Note that \eqref{int} holds for all small $t$ if and only if the homogeneous terms satisfy
\[
\bar{\partial}\phi_k = \frac{1}{2}\sum_{j=1}^k[\phi_j,\phi_{k-j}]
\qquad\text{for all } k\in\mathbb{N}.
\]

In particular, the deformations of complex structures on $X$ are unobstructed if and only if, for every $N\in\mathbb{N}$, any $N$-th order deformation $\sum_{1\le j\le N}\phi_j$ can be extended to order $N+1$, i.e.\ there exists $\phi_{N+1}$ such that $\sum_{1\le j\le N+1}\phi_j$ is an $(N+1)$-st order deformation.  
It is well known that the obstruction to extending an $N$-th order deformation $\sum_{1\le j\le N}\phi_j$ is given by
\[
\sum_{j=1}^N [\phi_j,\phi_{N+1-j}],
\]
which defines a cohomology class in $H_{\bar{\partial}}^2(X,T^{1,0})$.

\subsection{The cohomology contraction map}

Recall that the cohomology contraction map is defined by
\[
\mu \colon H^{0,2}_{\b{\p}}\bigl(X, T^{1,0}\bigr)
\longrightarrow 
\bigoplus_{p,q} \operatorname{Hom}\bigl(H_{\bar{\partial}}^{p,q}(X), H_{\bar{\partial}}^{p-1,q+2}(X)\bigr),
\quad
[\sigma] \longmapsto [i_\sigma(\bullet)],
\]
where, for each bidegree \((p,q)\), the map
\[
[i_\sigma(\bullet)] \colon H_{\bar{\partial}}^{p,q}(X) \longrightarrow H_{\bar{\partial}}^{p-1,q+2}(X)
\]
is induced by contraction with the vector-valued form \(\sigma\).  
Since \(\bar{\partial}\sigma=0\) implies \([\bar{\partial},i_\sigma]=0\), the operator \(i_\sigma\) preserves \(\bar{\partial}\)-closedness and hence is well-defined on Dolbeault cohomology.  
Moreover, if \(\sigma=\bar{\partial}\tau\), then
\[
i_\sigma = i_{\bar{\partial}\tau} = [\bar{\partial},i_\tau]
= \bar{\partial} i_\tau - i_\tau \bar{\partial},
\]
so \(i_\sigma\) induces the zero map on cohomology. Therefore \(\mu\) is a well-defined map on \(H_{\bar{\partial}}^2(X,T^{1,0})\).

For fixed \((p,q)\), we also consider the component
\[
\mu_{p,q} \colon 
H_{\bar{\partial}}^{0,2}(X,T^{1,0})
\longrightarrow
\operatorname{Hom}\bigl(H_{\bar{\partial}}^{p,q}(X), H_{\bar{\partial}}^{p-1,q+2}(X)\bigr),
\]
defined by
\(
\mu_{p,q}\colon [\sigma] \longmapsto [\,i_\sigma(\bullet)\,],
\)
so that
\[
\mu = \bigoplus_{p,q} \mu_{p,q}.
\]

We now recall the holomorphic Cartan homotopy formulas, which will be used frequently in what follows.  
For any \(\phi,\psi \in A^{0,1}\bigl(X,T^{1,0}\bigr)\), the (total) Lie derivative of \(\phi\) is defined by
\[
  \mathcal{L}_\phi := [i_\phi,d]
  = i_\phi \circ d - d \circ i_\phi
  = \mathcal{L}_\phi^{1,0} + \mathcal{L}_\phi^{0,1},
\]
where
\[
  \mathcal{L}_\phi^{1,0} := [i_\phi,\partial]
  = i_\phi\partial - \partial i_\phi,
  \qquad
  \mathcal{L}_\phi^{0,1} := [i_\phi,\bar{\partial}]
  = i_\phi\bar{\partial} - \bar{\partial} i_\phi
  = -\,i_{\bar{\partial}\phi}.
\]
Then the following identities hold:
\begin{equation}\label{Cartan-formula}
  [\mathcal{L}_\phi^{1,0}, i_\psi]
  = \mathcal{L}_\phi^{1,0} i_\psi - i_\psi \mathcal{L}_\phi^{1,0}
  = i_{[\phi,\psi]},
\end{equation}
and
\[
  [i_\phi,\bar{\partial}]
  = i_\phi\bar{\partial} - \bar{\partial} i_\phi
  = -\,i_{\bar{\partial}\phi},
\]
\[
  [\bar{\partial},\mathcal{L}_\phi^{1,0}]
  = \bar{\partial}\mathcal{L}_\phi^{1,0} + \mathcal{L}_\phi^{1,0}\bar{\partial}
  = \mathcal{L}_{\bar{\partial}\phi}^{1,0},
\]
\[
  [\mathcal{L}_\phi^{1,0},\mathcal{L}_\psi^{1,0}]
  = \mathcal{L}_\phi^{1,0}\mathcal{L}_\psi^{1,0}
   - \mathcal{L}_\psi^{1,0}\mathcal{L}_\phi^{1,0}
  = \mathcal{L}_{[\phi,\psi]}^{1,0}.
\]
These identities are standard (see, for example, \cite{MR3832143,MR2539771,FM06,Cle05}, where they are referred to as holomorphic Cartan homotopy formulas) and will be used without further comment.  
For completeness we refer the reader to \cite[Lemma~3.2]{LRY15}, \cite[Lemma~3.7]{Xia19deri}, and \cite[Lemma~3.3]{Xia19dDol}.  
As observed in \cite{Fri91} (see also \cite{LR11}), the Cartan formula \eqref{Cartan-formula} is closely related to the Tian--Todorov lemma \cite{Tia87,Tod89}.

\subsection{The contraction map annihilates obstructions}

In this section we show that the obstructions to deforming complex structures are contained in the kernel of the contraction map.

\begin{thm}[{\cite[Theorem~3.4]{LRY15}}]
For any $\phi\in A^{0,1}(X,T^{1,0})$ we have
\[
e^{-i_{\phi}}\circ\bar{\partial}\circ e^{i_\phi}
= \bar{\partial}-\mathcal{L}_\phi^{0,1},
\qquad
e^{-i_\phi}\circ\partial\circ e^{i_\phi}
= \partial-\mathcal{L}_\phi^{1,0}-i_{\frac{1}{2}[\phi,\phi]},
\]
where $\mathcal{L}_\phi^{1,0}=[i_\phi,\partial]$ and $\mathcal{L}_\phi^{0,1}=[i_\phi,\bar{\partial}]$.
\end{thm}

Set
\[
\bar{\partial}_\phi := \bar{\partial}+[\partial,i_\phi].
\]
Then
\begin{align*}
i_{\bar{\partial}\phi-\frac{1}{2}[\phi,\phi]}
&= e^{-i_\phi}\circ d\circ e^{i_\phi} - \partial - \bar{\partial}_\phi.
\end{align*}

Now write $\phi=\sum_{i=1}^N \phi_i$.  
Assume that
\[
\bigoplus_{r\ge 1} d_r^{p,q} = 0
\quad\text{and}\quad
\bigoplus_{\substack{r-1 \ge i\ge 0}} d_r^{p-1-i,\,q+1+i}(X)=0.
\]
By Theorem~\ref{thm11}, for any $\alpha_0^{p,q}\in A^{p,q}(X)\cap\ker\bar{\partial}$ the equation
\[
d\bigl(e^{i_\phi}(\alpha(t))\bigr) = 0 \ \text{mod } t^{N+1},
\qquad
\alpha(0)^{p,q} = \alpha_0^{p,q},
\]
admits a smooth solution $\alpha(t)=\sum_{k=0}^N\alpha_k\in F^pA^{p+q}(X)$.

Taking the homogeneous part of degree $N+1$ and looking at the $(p-1,q+2)$-component, we obtain
\begin{align}\label{eqn3.2}
\begin{split}
-\,i_{\frac{1}{2}[\phi,\phi]_{N+1}}\alpha_0^{p,q}
&= \bigl(i_{\bar{\partial}\phi-\tfrac{1}{2}[\phi,\phi]}\alpha(t)\bigr)^{p-1,q+2}_{N+1}\\
&= \bigl((e^{-i_\phi}-1)\circ d\circ e^{i_\phi}\,\alpha(t)\bigr)^{p-1,q+2}_{N+1}
   + \bigl(d\circ (e^{i_\phi}-1)\alpha(t)\bigr)^{p-1,q+2}_{N+1}\\
&\quad - \bigl([\partial,i_\phi]\alpha(t)\bigr)^{p-1,q+2}_{N+1}\\
&= \bigl(d\circ (e^{i_\phi}-1)\alpha(t)\bigr)^{p-1,q+2}_{N+1},
\end{split}
\end{align}
where in the last equality we used that 
\(
\bigl([\partial,i_\phi]\alpha(t)\bigr)^{p-1,q+2}_{N+1}=0.
\)

For the $(N+1)$-th step, Lemma~\ref{lemma2.2}(ii) gives
\[
d\bigl((e^{i_\phi}-1)\alpha\bigr)_{N+1}
= \Pi^{\ge p-1,*}\,d\bigl((e^{i_\phi}-1)\alpha\bigr)_{N+1}
\in dA^{p+q}(X)\cap F^{p-1}A^{p+q+1}(X).
\]
If, in addition,
\[
\bigoplus_{\substack{r-1 \ge i\ge 0}} d_r^{p-2-i,\,q+2+i}(X)=0,
\]
then by Lemma~\ref{cordeg2} there exists $\beta\in F^{p-1}A^{p+q}(X)$ such that
\[
d\bigl((e^{i_\phi}-1)\alpha\bigr)_{N+1} = d\beta.
\]
Substituting this into \eqref{eqn3.2} yields
\begin{align*}
-\,i_{\frac{1}{2}[\phi,\phi]_{N+1}}\alpha_0^{p,q}
&= \bigl(d\circ (e^{i_\phi}-1)\alpha(t)\bigr)^{p-1,q+2}_{N+1}\\
&= (d\beta)^{p-1,q+2}\\
&= \bar{\partial}\beta^{p-1,q+1}\in \mathrm{im}\,\bar{\partial}.
\end{align*}
By the definition of $\mu_{p,q}$, we conclude that
\[
\sum_{j=1}^N [\phi_j,\phi_{N+1-j}] \in \ker \mu_{p,q}.
\]

Moreover, the assumptions
\[
\bigoplus_{\substack{r-1 \ge i\ge 0}} d_r^{p-1-i,\,q+1+i}(X)=0
\quad\text{and}\quad
\bigoplus_{\substack{r-1 \ge i\ge 0}} d_r^{p-2-i,\,q+2+i}(X)=0
\]
are equivalent to
\[
\bigoplus_{r\ge 1} d_r^{p-1,\,q+1}(X)=0
\quad\text{and}\quad
\bigoplus_{\substack{r-1 \ge i\ge 0}} d_r^{p-2-i,\,q+2+i}(X)=0.
\]
Combining these observations, we obtain the following.

\begin{thm}\label{deformation obstruction}
Let $X$ be a compact complex manifold such that
\[
\bigoplus_{r\ge 1} d_r^{p,q} = 0,
\quad
\bigoplus_{r\ge 1} d_r^{p-1,q+1} = 0,
\quad
\bigoplus_{\substack{r-1 \ge i\ge 0}} d_r^{p-2-i,\,q+2+i}(X)=0.
\]
Then the obstructions
\(
\sum_{j=1}^N [\phi_j,\phi_{N+1-j}]
\)
lie in the kernel of the contraction map, i.e.
\[
\sum_{j=1}^N [\phi_j,\phi_{N+1-j}] \in \ker \mu_{p,q}.
\]
\end{thm}

Let $X$ be a compact complex manifold with trivial canonical bundle, $\dim X=n$, then there exists a nowhere vanishing holomorphic $(n,0)$-form $\Omega\in H^{n,0}_{\bar{\partial}}(X)$.  
Contraction with $\Omega$ induces an isomorphism
\[
\Omega\colon H^{0,2}_{\bar{\partial}}(X,T^{1,0})
\longrightarrow H^{n-1,2}_{\bar{\partial}}(X),
\qquad
[\sigma] \longmapsto [i_\sigma(\Omega)].
\]
For any $[\sigma]\in\ker\mu_{n,0}$ we have
\[
\mu_{n,0}([\sigma])(\Omega)
= [i_\sigma(\Omega)] = 0,
\]
and thus $[\sigma]=0$ by the above isomorphism. Hence
\[
\ker\mu_{n,0} = \{0\},
\]
and, in particular, by \eqref{d-operator} we have $d_r^{n,0}=0$ for all $r\ge 1$.

\begin{thm}\label{thm-CY}
Let $X$ be a compact complex manifold of dimension $n$ with trivial canonical bundle. Suppose that
\[
\bigoplus_{r\ge 1} d_r^{n-1,1} = 0
\quad\text{and}\quad
\bigoplus_{\substack{r-1 \ge i\ge 0}} d_r^{n-2-i,\,2+i}(X)=0.
\]
Then the deformations of the complex structure on $X$ are unobstructed.
\end{thm}

If, in addition, the Fr\"olicher spectral sequence of $X$ degenerates at $E_2$ in total degree $n$, i.e.\ $d_r^{p,q}=0$ for all $r\ge 2$ and all $p,q$ with $p+q=n$, we obtain the following corollary.

\begin{cor}\label{cor-CY}
Let $X$ be a compact complex manifold of dimension $n$ with trivial canonical bundle whose Fr\"olicher spectral sequence degenerates at $E_2$ in total degree $n$.  
If, moreover,
\[
d_1^{n-1,1}=0
\quad\text{and}\quad
d_1^{n-2,2}=0,
\]
then the deformations of the complex structure on $X$ are unobstructed.
\end{cor}
For solvable complex parallelisable
manifolds, the Fr\"olicher spectral sequence degenerates at $E_2$ \cite{Kas15,KS}. We have 
\begin{cor}\label{cor-CY-2}
Let $X$ be a $n$-dimensional solvable complex parallelisable manifold with
\[
d_1^{n-1,1}=0
\quad\text{and}\quad
d_1^{n-2,2}=0,
\]
then the deformations of the complex structure on $X$ are unobstructed.
\end{cor}

It is worth noting that Popovich \cite{MR3978322} also proved the following unobstructedness theorem.

\begin{thm}[{\cite[Observation~3.5]{MR3978322}}]\label{thm-po}
Let $X$ be a compact complex manifold of complex dimension $n$ with trivial canonical bundle $K_X$ such that the linear maps
\[
A_1\colon H_{\bar{\partial}}^{n-1,1}(X,\mathbb{C}) \longrightarrow H_{BC}^{n,1}(X,\mathbb{C}), 
\qquad
[\alpha]_{\bar{\partial}} \longmapsto [\partial\alpha]_{BC},
\]
and
\[
A_2\colon H_A^{n-2,2}(X,\mathbb{C}) \longrightarrow H_{BC}^{n-1,2}(X,\mathbb{C}),
\qquad
[v]_A \longmapsto [\partial v]_{BC},
\]
are identically zero. Then the Kuranishi family of $X$ is unobstructed.
\end{thm}

A simple argument shows that the conditions $A_1=0$ and $A_2=0$ are stronger than the vanishing assumptions on the differentials $d_r^{p,q}$ in Theorem~\ref{thm-CY}.  
In particular, Theorem~\ref{thm-CY} strictly generalizes Theorem~\ref{thm-po}.  
We now spell out the precise implications.

\begin{lemma}\label{lemma4.6}
We have:
\begin{itemize}
  \item[\emph{(i)}] If $A_1=0$, then $\displaystyle\bigoplus_{r\ge 1} d_r^{n-1,1} = 0$.
  \item[\emph{(ii)}] If $A_2=0$, then 
  \[
    \bigoplus_{\substack{r-1 \ge i \ge 0}} d_r^{n-2-i,\,2+i}(X)=0.
  \]
\end{itemize}
\end{lemma}

\begin{proof}
\emph{(i)} By definition of $A_1$, the condition $A_1=0$ is equivalent to the following: for every
\[
\alpha^{n-1,1}\in A^{n-1,1}(X)\cap\ker\bar{\partial},
\]
there exists $\beta^{n-1,0}\in A^{n-1,0}(X)$ such that
\[
\partial\alpha^{n-1,1} = \partial\bar{\partial}\beta^{n-1,0}.
\]
Set
\[
x := \alpha^{n-1,1} + \partial\beta^{n-1,0} \in F^{n-1}A^n(X).
\]
Then $\Pi^{n-1,1}x = \alpha^{n-1,1}$ and
\[
dx = \partial\alpha^{n-1,1} + \bar{\partial}\partial\beta^{n-1,0} = 0,
\]
so $x\in \ker d \cap F^{n-1}A^n(X)$.  
By Lemma~\ref{cordeg1}, we have
\[
\bigoplus_{r\ge 1} d_r^{n-1,1}=0.
\]

\medskip

\emph{(ii)} For $A_2=0$, it is shown in \cite[Page~688]{MR3978322} that this is equivalent to the following statement:  
for every
\[
\beta^{n-1,2}\in A^{n-1,2}(X)\cap\ker d \cap \mathrm{im}\,\partial,
\]
there exists $\gamma^{n-2,1}\in A^{n-2,1}(X)$ such that
\[
\beta^{n-1,2} = \partial\bar{\partial}\gamma^{n-2,1}.
\]
By Lemma~\ref{cordeg2}, it suffices to prove the inclusion
\begin{equation}\label{eqn5}
F^{n-1}A^{n+1}(X)\cap dA^n(X)\subset dF^{n-1}A^n(X).
\end{equation}

Let $d\alpha\in F^{n-1}A^{n+1}(X)\cap dA^n(X)$, where
\[
\alpha = \sum_{p+q=n}\alpha^{p,q}\in A^n(X).
\]
Since $d\alpha\in F^{n-1}A^{n+1}(X)$, inspecting its $(\leq (n-2),*)$-component yields
\[
\bar{\partial}\alpha^{n-2,2} + d\Bigl(\sum_{p=1}^{n-3}\alpha^{p,n-p}\Bigr)=0,
\]
and in particular
\[
\bar{\partial}\alpha^{n-2,2} = -\,\partial\alpha^{n-3,3}.
\]
Thus
\[
\partial\alpha^{n-2,2}\in A^{n-1,2}(X)\cap\ker d \cap \mathrm{im}\,\partial.
\]
If $A_2=0$, there exists $\gamma^{n-2,1}\in A^{n-2,1}(X)$ such that
\[
\partial\alpha^{n-2,2} = \partial\bar{\partial}\gamma^{n-2,1}.
\]
Hence
\begin{align*}
d\alpha
&= d(\alpha^{n,0}+\alpha^{n-1,1}) + \partial\alpha^{n-2,2} \\
&= d\bigl(\alpha^{n,0}+\alpha^{n-1,1} - \partial\gamma^{n-2,1}\bigr)
  \in dF^{n-1}A^n(X),
\end{align*}
and so the inclusion \eqref{eqn5} holds.  
The claim follows from Lemma~\ref{cordeg2}.
\end{proof}

Consequently, Theorem~\ref{thm-po} is an immediate consequence of Theorem~\ref{thm-CY} together with Lemma \ref{lemma4.6}.

\begin{ex}\label{ex-Naka}
We consider the Nakamura manifold \cite{Naka} associated with the six-dimensional solvable Lie algebra \(\mathfrak{s}_{12}\), endowed with a splitting-type complex structure.  
Specifically, we take the splitting-type structures \(J_C\) appearing in \cite[(3.2)]{AOUV} with parameter \(C=\frac{i}{2k}\), where \(0\neq k\in\mathbb{Z}\).  
The structure equations are
\begin{equation}\label{eqn4.1}
J_C:\left\{
\begin{aligned}
d\omega_C^{1} &= -(C-i)\,\omega_C^{13}-(C+i)\,\omega_C^{1\bar3},\\
d\omega_C^{2} &= (C-i)\,\omega_C^{23}+(C+i)\,\omega_C^{2\bar3},\\
d\omega_C^{3} &= 0,
\end{aligned}\right.
\end{equation}
where \(\omega^{i\bar j} := \omega^i\wedge\bar{\omega}^j\), and \(\{\omega^1,\omega^2,\omega^3\}\) is a global holomorphic coframe.  
By \cite[Remark~3.2]{AOUV}, the pair \((\mathfrak{s}_{12},J_C)\) determines a complex solvmanifold \(X\) with holomorphically trivial canonical bundle.  
Moreover, one has \[\dim H^{0,1}_{\bar\partial}(X,T^{1,0})=\dim H^{2,1}_{\bar\partial}(X)=5,\] see \cite[Table~10]{AOUV}.

The connected, simply-connected solvable Lie group with Lie algebra \(\mathfrak{s}_{12}\) is given by
\[
G=\mathbb{C}\ltimes_{\varphi_C}\mathbb{C}^2,
\]
equipped with the left-invariant complex structure \(J_C\) defined by \eqref{eqn4.1}.  
The action \(\varphi_C\) is diagonal, with characters
\[
\alpha_1^C(z_3)=e^{-(C-i)z_3-(C+i)\bar z_3},\qquad 
\alpha_2^C(z_3)=\bigl(\alpha_1^C(z_3)\bigr)^{-1}.
\]
Let \(\Gamma=\Gamma'\ltimes_{\varphi_C}\Gamma''\) be a lattice in \(G\), where
\[
\Gamma'=\frac{\pi}{2\,\mathrm{Im}(C)}\bigl(1-i\,\mathrm{Re}(C)\bigr)\mathbb{Z}
\;\oplus\;
\frac{i}{2}\log\!\left(\frac{3+\sqrt5}{2}\right)\mathbb{Z},
\]
and \(\Gamma''\subset \mathbb{C}^2\) is a lattice satisfying  
\(\varphi_C(z_3)(\Gamma'')\subset\Gamma''\) for all \(z_3\in\Gamma'\).  
Then \(X=G/\Gamma\) is a complex solvmanifold of splitting type whose underlying real Lie algebra is \(\mathfrak{s}_{12}\).

It is known that the Fr\"olicher spectral sequence of \(X\) degenerates at \(E_2\) (in fact,\ \(X\) satisfies the page-\(1\;\partial\bar\partial\) property), see \cite[Theorem~A]{KS}.

Recall that \(E_1^{p,q}=H^{p,q}_{\bar\partial}(X)\), and that
\[
d_1^{p,q}\colon E_1^{p,q}\longrightarrow E_1^{p+1,q},\qquad
d_1^{p,q}[\alpha]=[\partial\alpha].
\]
From \cite[Table~10]{AOUV} the Dolbeault cohomology groups and their generators are known explicitly.  
Using \cite[(3.6)--(3.7)]{AOUV}, we adopt the coframe
\[
\left\{
\begin{aligned}
&\varphi^1 = e^{(C+\bar C-2i) z_3}dz_1,\qquad
\varphi^2 = e^{-(C+\bar C-2i) z_3}dz_2,\qquad
\varphi^3 = dz_3,\\[2pt]
&\tilde\varphi^1 = e^{(C+\bar C-2i) z_3}d\bar z_1,\qquad
\tilde\varphi^2 = e^{-(C+\bar C-2i) z_3}d\bar z_2,\qquad
\tilde\varphi^3 = d\bar z_3,
\end{aligned}
\right.
\]
where the first triple has bidegree \((1,0)\) and the second \((0,1)\).  
In this basis, the structure equations become
\[
\left\{
\begin{aligned}
d\varphi^1 &= -(C+\bar C-2i)\,\varphi^{13},\\
d\varphi^2 &=  (C+\bar C-2i)\,\varphi^{23},\\
d\varphi^3 &= 0,
\end{aligned}
\qquad
\begin{aligned}
d\tilde\varphi^1 &= (C+\bar C-2i)\,\varphi^{3\tilde1},\\
d\tilde\varphi^2 &= -(C+\bar C-2i)\,\varphi^{3\tilde2},\\
d\tilde\varphi^3 &= 0.
\end{aligned}
\right.
\]

We use the wedge shorthand \(\varphi^{ij}=\varphi^i\wedge\varphi^j\), \(\varphi^{ijk}=\varphi^i\wedge\varphi^j\wedge\varphi^k\), and similarly for mixed forms.

\begin{itemize}
\item[\(\boldsymbol{E_1^{3,1}}\).]
The group \(E^{3,1}_1\) is generated by \([\varphi^{123\tilde3}]\).  
Using the expressions for \(\varphi^i\), we have
\[
\varphi^{123\tilde3} = dz_1\wedge dz_2\wedge dz_3\wedge d\bar z_3,
\]
so \(\partial\varphi^{123\tilde3}=0\).  
Hence \(d_1^{3,1}=0\).

\item[\(\boldsymbol{E_1^{2,2}}\).]
This group is generated by
\[
\mathbb{C}\langle
\varphi^{12\tilde1\tilde2},\,
\varphi^{13\tilde1\tilde3},\,
\varphi^{13\tilde2\tilde3},\,
\varphi^{23\tilde1\tilde3},\,
\varphi^{23\tilde2\tilde3}
\rangle.
\]
The form \(\varphi^{12\tilde1\tilde2}=dz_1\wedge dz_2\wedge d\bar z_1\wedge d\bar z_2\) lies in \(\ker\partial\).  
All other generators contain \(dz_3\), and are likewise \(\partial\)-closed.  
Thus \(d_1^{2,2}=0\).

\item[\(\boldsymbol{E_1^{1,3}}\).]
The group \(E^{1,3}_1\) is generated by \([\overline{\varphi^{123\tilde3}}]\), which is also \(\partial\)-closed.  
Hence \(d_1^{1,3}=0\).
\end{itemize}

Consequently,
\[
d_1^{p,q}=0 \qquad\text{whenever } p+q=4.
\]
Since the Fr\"olicher spectral sequence of \(X\) degenerates at \(E_2\), it follows that all differentials on \(p+q=4\) vanish for every \(r\ge1\).

On the other hand,
\[
d_1^{1,1}\bigl([\varphi^{1\tilde1}]\bigr)
= [\partial\varphi^{1\tilde1}]
= [-2(C+\bar C-2i)\,\varphi^{13\tilde1}]
\neq 0,
\]
so the spectral sequence does not degenerate at \(E_1\).

Finally, since 
\[
H^{3,1}_{\bar\partial}(X)=\mathbb{C}[\varphi^{123\tilde3}],\qquad
H^{2,3}_{\bar\partial}(X)=\mathbb{C}[\varphi^{12\tilde1\tilde2\tilde3}],
\]
and
\[
\mu_{3,1}\colon 
H^{0,2}_{\bar\partial}(X,T^{1,0})\cong H^{2,2}_{\bar\partial}(X)
\longrightarrow 
\mathrm{Hom}\bigl(H^{3,1}_{\bar\partial}(X), H^{2,3}_{\bar\partial}(X)\bigr),
\]
with
\[
H^{2,2}_{\bar\partial}(X)=
\mathbb{C}\langle
[\varphi^{12\tilde1\tilde2}],
[\varphi^{13\tilde1\tilde3}],
[\varphi^{13\tilde2\tilde3}],
[\varphi^{23\tilde1\tilde3}],
[\varphi^{23\tilde2\tilde3}]
\rangle,
\]
we find
\begin{equation}\label{eqn4.4}
\ker\mu_{3,1}
=
\mathbb{C}\bigl\langle
[\varphi^{13\tilde1\tilde3}],
[\varphi^{13\tilde2\tilde3}],
[\varphi^{23\tilde1\tilde3}],
[\varphi^{23\tilde2\tilde3}]
\bigr\rangle.
\end{equation}

In fact, one may directly check the case \(p+q=1\), i.e.\ verify that \(d_1^{0,1}=d_1^{1,0}=0\).  
This is immediate since
\[
H^{0,1}_{\bar\partial}(X)=\mathbb{C}[\tilde\varphi^3],
\qquad
H^{1,0}_{\bar\partial}(X)=\mathbb{C}[\varphi^3].
\]
Moreover, one easily sees that \(\ker\mu_{1,0}\) is given precisely by \eqref{eqn4.4}. By Theorem \ref{deformation obstruction}, the obstructions are contained in $\ker \mu_{3,1}\cap \ker \mu_{1,0}=\ker \mu_{3,1}$.

On the other hand, it is known that the Kuranishi space of the Nakamura manifold is a quadratic cone and thus not smooth \cite{Naka,GM88}. This may be illustrated by our Corollary \ref{cor-CY-2} because in this case both conditions $d_1^{n-1,1}=0, d_1^{n-2,2}=0$ are not satisfied.
\end{ex}

\subsection{Kodaira principles}\label{sec-KP}

In his monograph \cite{Man}, Manetti studied obstructions to deformations of complex structures and, under suitable hypotheses, proved that these obstructions lie in the kernel of the cohomological contraction map. These results are referred to as Kodaira principles. In this subsection, we show that they follow directly from Theorem~\ref{deformation obstruction}. In this sense, Theorem~\ref{deformation obstruction} may be viewed as a refined Kodaira principle.

The following statement is \cite[Theorem~8.6.1]{Man}, usually called the Kodaira principle.

\begin{thm}[Kodaira principle]
Let $X$ be a complex manifold such that the subcomplex $\partial A_X^{*,*}$ of $\partial$-exact forms is acyclic. Then the contraction map in Dolbeault cohomology
\[
\boldsymbol{i}\colon
H^2\left(X,T^{1,0}\right) \longrightarrow
\bigoplus_{p,q}\operatorname{Hom}\bigl(H^p(X,\Omega_X^q),\,H^{p+2}(X,\Omega_X^{q-1})\bigr),
\quad
\boldsymbol{i}_\eta(\omega)=\eta\lrcorner\omega,
\]
annihilates every obstruction to deformations of $X$.
\end{thm}

\begin{proof}
The acyclicity of the subcomplex $\partial A_X^{*,*}$ is equivalent to
\begin{equation}\label{eq:acyclic-dbar-d}
  \ker\bar\partial \cap \operatorname{im}\partial
  \;=\;
  \operatorname{im}\partial\bar\partial.
\end{equation}
This condition is stronger than the degeneration of the Fr\"olicher spectral sequence at $E_1$.

Indeed, let $\alpha^{p,q}\in A^{p,q}(X)\cap\ker\bar\partial$. Then
\(
\partial\alpha^{p,q}\in \ker\bar\partial\cap\operatorname{im}\partial
\),
so by \eqref{eq:acyclic-dbar-d} there exists $\beta^{p,q-1}\in A^{p,q-1}(X)$ such that
\[
\partial\alpha^{p,q}
=
\partial\bar\partial\beta^{p,q-1}.
\]
Set
\[
\alpha \;:=\; \alpha^{p,q} + \partial\beta^{p,q-1} \in F^pA^{p+q}(X).
\]
Then $\Pi^{p,q}\alpha=\alpha^{p,q}$ and
\[
d\alpha
=
\bar\partial\alpha^{p,q}
+
\bigl(\partial\alpha^{p,q}+\bar\partial\partial\beta^{p,q-1}\bigr)
=
0.
\]
Hence, for all $p,q$ and every $r\ge1$ we have
\[
d_r^{p,q}=0,
\]
i.e. the Fr\"olicher spectral sequence degenerates at $E_1$.  
Applying Theorem~\ref{deformation obstruction}, we conclude that all obstructions to deformations of $X$ lie in the kernel of the contraction map, which is exactly the assertion of the Kodaira principle.
\end{proof}

Next, we interpret our filtration conditions in terms of injectivity in de~Rham cohomology.  
Note that the identity
\[
F^{p+1}A^{p+q+1}(X)\cap dA^{p+q}(X)
    \;=\;
    dF^{p+1}A^{p+q}(X)
\]
is equivalent to the injectivity of the induced map
\[
\frac{\ker d\cap F^{p+1}A^{p+q+1}(X)}
     {d\bigl(F^{p+1}A^{p+q}(X)\bigr)}
\;\longrightarrow\;
\frac{\ker d\cap A^{p+q+1}(X)}
     {dA^{p+q}(X)}.
\]
Similarly, the condition
\[
F^{p+1}A^{p+q+1}(X)\cap dF^{p}A^{p+q}(X)
    \;=\;
    dF^{p+1}A^{p+q}(X)
\]
is equivalent to the injectivity of
\[
\frac{\ker d\cap F^{p+1}A^{p+q+1}(X)}
     {d\bigl(F^{p+1}A^{p+q}(X)\bigr)}
\;\longrightarrow\;
\frac{\ker d\cap F^{p}A^{p+q+1}(X)}
     {d\bigl(F^{p}A^{p+q}(X)\bigr)}.
\]

For convenience, set
\[
F_X^p
:=
\bigoplus_{\substack{i\ge p \\ 0\le q\le n}} A^{i,q}(X)
=
\bigoplus_{0\le q\le n} F^pA^{p+q}(X).
\]
In particular, $F_X^0=A^{*,*}(X)$. We then obtain the following two corollaries.

\begin{cor}\label{cor4.3}
The following two conditions are equivalent:
\begin{enumerate}
  \item The natural map \(F_X^{p+1}\to F_X^{0}\) is injective in cohomology.
  \item For every \(q\), the identity
  \[
  F^{p+1}A^{p+q+1}(X)\cap dA^{p+q}(X)
      \;=\;
      dF^{p+1}A^{p+q}(X)
  \]
  holds.
\end{enumerate}
\end{cor}

\begin{cor}
The following two conditions are equivalent:
\begin{enumerate}
  \item The natural map \(F_X^{p+1}\to F_X^{p}\) is injective in cohomology.
  \item For every \(q\), one has
  \[
  F^{p+1}A^{p+q+1}(X)\cap dF^pA^{p+q}(X)
      \;=\;
      dF^{p+1}A^{p+q}(X).
  \]
\end{enumerate}
\end{cor}

Combining these observations with Theorem~\ref{deformation obstruction}, using Lemma \ref{cordeg1} and \ref{cordeg2}, we recover the following version of the Kodaira principle.

\begin{thm}[{\cite[Corollary~8.8.4]{Man}, Kodaira principle}]
Let $X$ be a complex manifold and let $p$ be a positive integer such that the three inclusions
\[
F_X^{p+1} \longrightarrow F_X^p \longrightarrow F_X^{p-1} \longrightarrow F_X^0 = A_X^{*,*}
\]
are injective in cohomology. Then the contraction map
\[
\boldsymbol{i}\colon
H^2\left(X,T^{1,0}\right)
\longrightarrow
\bigoplus_{q=0}^n\operatorname{Hom}\bigl(H^q(X,\Omega_X^p),\,H^{q+2}(X,\Omega_X^{p-1})\bigr),
\quad
\boldsymbol{i}_\eta(\omega)=\eta\lrcorner\omega,
\]
annihilates every obstruction to deformations of $X$.
\end{thm}

\section{Deformations of complex parallelisable manifolds}\label{sec-para}
Compact complex manifolds with trivial holomorphic tangent bundle are often called
\emph{complex parallelisable manifolds} in the literature, see for instance \cite{Naka}.  
By a classical theorem of Wang \cite{MR74064}, such a manifold must be of the form 
\(\Gamma \backslash G\), where \(G\) is a simply connected, connected complex Lie group and 
\(\Gamma \subset G\) is a cocompact lattice.  
In \cite{Naka}, Nakamura studied complex parallelisable manifolds and their small deformations; 
later, in \cite{Rol11a}, Rollenske analysed the Kuranishi space of complex parallelisable 
nilmanifolds, see also \cite{PRW24} and \cite[Sec.\,6]{GM90}.  
The standard example is the Iwasawa manifold, whose deformations are unobstructed.  
We emphasize that the Iwasawa manifold is rather special: its Fr\"olicher spectral sequence 
degenerates at \(E_2\) but not at \(E_1\), and one has \(d_1^{p,q}\neq 0\) if and only if 
\(p=1\); see, for example, \cite[Section~5.1.2]{WX2023}.  
Consequently, our Theorem~\ref{thm-CY} cannot be applied directly to verify unobstructed 
deformations for the Iwasawa manifold.  
In this subsection we therefore continue to use analytic methods to study unobstructed 
deformations for this class of manifolds.

Let \(X\) be a compact complex manifold with trivial holomorphic tangent bundle, and let
\(\{\theta_1,\dots,\theta_n\}\) denote a global holomorphic frame of \(T^{1,0}_X\).  
Since \([\theta_j,\theta_l]\) is again a holomorphic vector field, we may write
\[
  [\theta_j,\theta_l] = \sum_{p=1}^n C^p_{jl}\,\theta_p
\]
for some constants \(C^p_{jl}\).

For any class \([\phi_1]\in H^{0,1}_{\bar{\partial}}(X,T^{1,0})\), we can write
\[
  \phi_1 = \sum_{i,j} t_{ij}\,\psi_i\,\theta_j,
\]
where \([\psi_i]\in H^{0,1}_{\bar{\partial}}(X)\).  
If all holomorphic \((0,1)\)-forms are \(d\)-closed, then for any 
\(\phi_1,\phi_2\in \ker\bar{\partial}\cap A^{0,1}(X,T^{1,0})\) we have
\[
  [\phi_1,\phi_2]
  = \sum_{i,j,k,l} t_{ij}t'_{kl}\,[\psi_i\theta_j,\psi_k\theta_l]
  = \sum_{i,j,k,l} t_{ij}t'_{kl}\,\psi_i\wedge\psi_k\,[\theta_j,\theta_l].
\]
We set
\[
  H^{0,2}_d(X):=\frac{\ker d\cap A^{0,2}(X)}{dA^{0,1}(X)\cap A^{0,2}(X)}.
\]
Assume that
\begin{equation*}
  [H^{0,1}_{\bar{\partial}}(X),H^{0,1}_{\bar{\partial}}(X)] = 0
  \quad\text{in }H^{0,2}_d(X).
\end{equation*}
Then for each pair \(i,k\) we can write
\[
  \psi_i\wedge\psi_k = d\psi_{ik} = \bar{\partial}\psi_{ik}
\]
for some \(\psi_{ik}\in A^{0,1}(X)\cap \ker\partial\).  
It follows that
\[
  [\phi_1,\phi_1]
  = \sum_{i,j,k,l} t_{ij}t_{kl}\,\psi_i\wedge\psi_k\,[\theta_j,\theta_l]
  = \bar{\partial}\!\left(\sum_{i,j,k,l} t_{ij}t_{kl}\,\psi_{ik}\,[\theta_j,\theta_l]\right).
\]
Hence we may choose
\[
  \phi_2 = \frac{1}{2}\sum_{i,j,k,l} t_{ij}t_{kl}\,\psi_{ik}\,[\theta_j,\theta_l]
  = \frac{1}{2}\sum_{p=1}^n
     \Bigl(\sum_{i,j,k,l} t_{ij}t_{kl}\,\psi_{ik}\,C^p_{jl}\Bigr)\theta_p,
\]
where 
\(\sum_{i,j,k,l} t_{ij}t_{kl}\,\psi_{ik}\,C^p_{jl} \in A^{0,1}(X)\cap \ker \partial\).

Proceeding inductively, suppose that \(\phi_1,\dots,\phi_N\) have been constructed, each of the form
\[
  \phi_i = \sum_{j=1}^n \alpha_{ij}\,\theta_j, 
  \qquad \alpha_{ij}\in A^{0,1}(X)\cap \ker\partial.
\]
Then
\begin{align*}
  \sum_{i+j=N+1} [\phi_i,\phi_j]
  &= \sum_{i+j=N+1} \sum_{k,l} [\alpha_{ik}\theta_k,\,\alpha_{jl}\theta_l] \\
  &= \sum_{i+j=N+1} \sum_{k,l} \alpha_{ik}\wedge\alpha_{jl}\,[\theta_k,\theta_l] \\
  &= \sum_{p=1}^n\!\left(\sum_{i+j=N+1} \sum_{k,l} 
        \alpha_{ik}\wedge\alpha_{jl}\,C^p_{kl}\right)\theta_p 
     \in \ker \bar{\partial},
\end{align*}
since \(\theta_k(\alpha_{jl}) = \theta_k\lrcorner(\partial \alpha_{jl}) = 0\).  
Therefore each coefficient
\[
  \sum_{i+j=N+1}\sum_{k,l}\alpha_{ik}\wedge\alpha_{jl}\,C^p_{kl}
\]
is a holomorphic \((0,2)\)-form.  
Because \(\alpha_{ij}\in A^{0,1}(X)\cap \ker\partial\), we have
\[
  \sum_{i+j=N+1}\sum_{k,l}\alpha_{ik}\wedge\alpha_{jl}\,C^p_{kl}
  \in \ker d \cap A^{0,2}(X).
\]
If, moreover, \(H^{0,2}_d(X)=\{0\}\), then there exist \(\beta_p\in A^{0,1}(X)\) such that
\[
  \sum_{i+j=N+1}\sum_{k,l}\alpha_{ik}\wedge\alpha_{jl}\,C^p_{kl}
  = d\beta_p = \bar{\partial}\beta_p.
\]
Thus we can define
\[
  \phi_{N+1} = \frac{1}{2}\sum_{p=1}^n \beta_p\,\theta_p,
\]
with \(\beta_p \in A^{0,1}(X)\cap \ker\partial\).  
In this way, the inductive construction yields unobstructed deformations.

If in addition
\[
  [[H^{0}_{\bar{\partial}}(X,T^{1,0}),H^{0}_{\bar{\partial}}(X,T^{1,0})],
   H^{0}_{\bar{\partial}}(X,T^{1,0})]=0,
\]
then \([\phi_2,\phi_1]=0\), and we may take \(\phi_i=0\) for all \(i\geq 3\).

We thus obtain the following results.

\begin{prop}
Let \(X\) be a compact complex manifold with trivial holomorphic tangent bundle such that all holomorphic \((0,1)\)-forms are \(d\)-closed.  
If moreover \(H^{0,2}_d(X)=\{0\}\), then \(X\) has unobstructed deformations.
\end{prop}

\begin{prop}\label{prop12}
Let \(X\) be a compact complex manifold with trivial holomorphic tangent bundle such that all holomorphic \((0,1)\)-forms are \(d\)-closed and
\(
  [H^{0,1}_{\bar{\partial}}(X),H^{0,1}_{\bar{\partial}}(X)]
\)
has trivial class in \(H^{0,2}_d(X)\).  
If, in addition,
\[
  [[H^{0}_{\bar{\partial}}(X,T^{1,0}),H^{0}_{\bar{\partial}}(X,T^{1,0})],
   H^{0}_{\bar{\partial}}(X,T^{1,0})]=0,
\]
then \(X\) has unobstructed deformations.
\end{prop}

As an application, we recover the classical fact that the Iwasawa manifold has unobstructed deformations (see \cite{Naka} for further details).  

Let \(X=\mathbb{C}^3/\Gamma\) be the Iwasawa manifold, where \(g\in\Gamma\) acts on \(\mathbb{C}^3\) by
\[
  z'_1=z_1+\omega_1,\qquad
  z'_2=z_2+\omega_2,\qquad
  z'_3=z_3+\omega_1 z_2+\omega_3,
\]
for \(g=(\omega_1,\omega_2,\omega_3)\) and \(z'=z\cdot g\).  
There exist holomorphic \(1\)-forms \(\varphi_1,\varphi_2,\varphi_3\) which are pointwise linearly independent on \(X\) and given by
\[
  \varphi_1=dz_1,\qquad
  \varphi_2=dz_2,\qquad
  \varphi_3=dz_3-z_1dz_2,
\]
so that
\[
  d\varphi_1=d\varphi_2=0,\qquad
  d\varphi_3=\partial\varphi_3=-\varphi_1\wedge\varphi_2.
\]
On the other hand, there are holomorphic vector fields \(\theta_1,\theta_2,\theta_3\) on \(X\) given by
\[
  \theta_1=\partial_1,\qquad
  \theta_2=\partial_2+z_1\partial_3,\qquad
  \theta_3=\partial_3,
\]
where \(\partial_\lambda=\partial/\partial z_\lambda\).  
It is straightforward to check that
\[
  [\theta_1,\theta_2]=-[\theta_2,\theta_1]=\theta_3,
  \qquad
  [\theta_2,\theta_3]=[\theta_1,\theta_3]=0.
\]
The space \(H^{0,1}_{\bar{\partial}}(X)\) is spanned by \([\bar{\varphi}_1]\) and \([\bar{\varphi}_2]\).  
From the definitions of \(\varphi_1\) and \(\varphi_2\), we see that all holomorphic \((0,1)\)-forms are \(d\)-closed.

Moreover, \([H^{0,1}_{\bar{\partial}}(X),H^{0,1}_{\bar{\partial}}(X)]\) is one-dimensional, generated by 
\([\bar{\varphi}_1\wedge\bar{\varphi}_2]\).  
Since
\[
  \bar{\varphi}_1\wedge\bar{\varphi}_2
  = d(-\bar{\varphi}_3)\in dA^{0,1}(X)\cap A^{0,2}(X),
\]
its class vanishes in \(H^{0,2}_d(X)\).  
Thus Proposition~\ref{prop12} applies, and we conclude that the Iwasawa manifold has unobstructed deformations.

\bibliographystyle{alpha}
\bibliography{deformation}

\end{CJK}
\end{document}